\documentclass[12pt]{article}
\usepackage{times}
\usepackage{amsmath,amssymb}
\usepackage{theoremref}
\usepackage[utf8]{inputenc}
\usepackage{indentfirst}
\usepackage{graphicx,tikz}
\usepackage{xcolor}
\usepackage{xskak}
\usepackage[T1]{fontenc}
\usepackage{amsthm}
\usepackage{indentfirst}

\usepackage{parskip}
\newtheorem{problem}{Question}%[section]

\newtheorem{theorem}[problem]{Theorem}

\newtheorem{lemma}[problem]{Lemma}

\title{Confirmation of the Daykin-Frankl Conjecture}
\author{Kada Kálmán Williams}

\begin{document}

\maketitle

\begin{abstract}
In 1983, Daykin and Frankl conjectured that if $P$ is a convex subset of $Q_n$, then it contains at least $|P|\binom{n}{\lfloor n/2\rfloor}2^{-n}$ pairwise incomparable elements. We verify and communicate an LLM-generated proof of this conjecture.
\end{abstract}

\pagenumbering{arabic}
\setlength{\baselineskip}{18pt}

\section{Introduction}

Let $(P,\le )$ be a poset an $S\subseteq P$. We say that:
\begin{itemize}
    \item $S$ is an upset if whenever $x\in S$ and $x<y$, also $y\in S$,
    \item $S$ is a downset if $S^c$ is an upset, i.e. whenever $y\in S$ and $x<y$, also $x\in S$,
    \item $S$ is a convex set if it is the intersection of an upset and a downset, i.e. whenever $u\le v\le w$ and $u,w\in S$, also $v\in S$.
\end{itemize}
An antichain in $P$ is a subset of pairwise incomparable elements. The width $w(P)$ of a finite poset $P$ is the maximum size of an antichain in $P$.

Let $Q_n$ be the power set of a set with $n$ elements, ordered by the subset relation. Daykin and Frankl \cite{DFr} conjectured that if $P\subseteq Q_n$ is a convex set, then
$$\frac{w(P)}{|P|}\ge \frac{w(Q_n)}{|Q_n|}.$$
This has previously been proven if $P$ is a "binary downset" \cite{DHL}, not in general.

For a wider discussion of this topic, see the introduction of Chapter 2 and the proof of the Harris-Kleitman inequality in Chapter 1 of the author's doctoral thesis \cite{Wil}.

The strict aim of this note is to prove the Daykin-Frankl conjecture.

\section{Outline of the Proof}

Recall that if $P$ and $Q$ are posets, $P\times Q$ is the set of ordered pairs $(p,q)$ with $p\in P$ and $q\in Q$, ordered by $(p_1,q_1)\le (p_2,q_2)$ if $p_1\le p_2$ and $q_1\le q_2$ \cite{Wil}. Note that $Q_{n+k}\cong Q_n\times Q_k$ and that the product of convex sets is convex.

We prove a strengthening of the Daykin-Frankl conjecture by backward induction:

\begin{theorem} \label{main}
    Let $n,k\ge 0$ be integers. Suppose $P\subseteq Q_n$ is a convex subset of $Q_n$, ordered by the subset relation. Then
    $$w(P\times Q_k)\ge w(Q_{n+k})|P|2^{-n}.$$
\end{theorem}

The induction step is guaranteed by the following key lemma:

\begin{lemma} \label{key}
    Let $R$ be a poset and $P\subseteq R\times Q_1$ be a convex subset. We identify $Q_1$ as $\{0,1\}$ and let $P=P_0\times \{0\}\cup P_1\times \{1\}$, where $P_0,P_1\subseteq R$. Then
    $$w(P)\ge \frac{w(P_0\times Q_1)+w(P_1\times Q_1)}{2}.$$
\end{lemma}

\begin{proof}[Proof of Theorem \ref{main} via Lemma \ref{key}]
    We induct on $n$. If $n=0$, $Q_0=\{\emptyset\}$, so $P=\emptyset$ or $P=Q_0$, with $w(P\times Q_k)=0$ or $=w(Q_k)$, respectively, as claimed. Let us now prove the result for $n+1$, given that it holds for $n$.

    Let $P\subseteq Q_{n+1}\cong Q_n\times Q_1$ be convex and $P=P_0\times \{0\}\cup P_1\times\{1\}$. Then also $P\times Q_k\subseteq Q_{n+1}\times Q_k\cong Q_{n+k}\times Q_1$ is convex. By Lemma \ref{key}, with $R=Q_{n+k}$,
    $$w(P\times Q_k)\ge \frac{w(P_0\times Q_{k+1})+w(P_1\times Q_{k+1})}{2}.$$
    Since $P_0,P_1\subseteq Q_n$ are convex, the result follows by induction.    
\end{proof}

\section{Proof of Key Lemma}

\begin{proof}[Proof of Lemma \ref{key}]
Consider an antichain of size $w(P_i\times Q_1)$ in $P_i\times Q_1$ and partition it as $L_i\times \{0\}\cup H_i\times \{1\}$, where $H_i,L_i\subseteq P_i$ ($i=0,1$). Then
$$|L_0|+|L_1|+|H_0|+|H_1|=w(P_0\times Q_1)+w(P_1\times Q_1).$$
If we specify two antichains in $P$ with sum of sizes equal to this, we are done.

Observe that if $x\in P_0$ and $y\in P_1$ with $x\le y$, then $(x,0),(y,1)\in P$. Since $P$ is convex, this implies $(x,1),(y,0)\in P$, and so $x,y\in P_0\cap P_1=:B$ ($\star$).

We sort elements of $L_1\cup H_0$ in $B$ into the sets $L$ and $H$ as follows:
\begin{itemize}
    \item If $x\in L_1\cap B$, sort $x$ into $L$ if there is no $z\in L_0\cup (H_0\cap C)$ such that $z\le x$. Otherwise, $x$ is sorted into $H$.
    \item If $y\in H_0\cap B$, sort $y$ into $H$ if there is no $z\in H_1\cup (L_1\cap C)$ such that $y\le z$. Otherwise, $y$ is sorted into $L$.
\end{itemize}

To finish, it clearly suffices to establish that the following are antichains:
$$L_0\times\{0\}\cup (L_1\setminus B\cup L)\times \{1\},\qquad H_1\times\{1\}\cup (H_0\setminus B\cup H)\times \{0\}.$$
If the former is an antichain, then so is the latter, by hourglass symmetry.

Suppose, for the sake of contradiction, that some two elements in $L_0\times\{0\}$ and $(L_1\setminus B\cup L)\times \{1\}$ are comparable. If $z\in L_0$, $x\in L_1\setminus B\cup L$, and $z\le x$:
\begin{itemize}
    \item if $x\in H_0$, then $(z,0)\le (x,1)$, contradicting the antichain property,
    \item if $x\in L_1$, then by ($\star$), $x\in B$, so $x\in L$, contradicting the definition of $L$.
\end{itemize}
Next, $L_1\setminus B\cup L=(L_1\setminus B)\cup (L_1\cap L)\cup (H_0\cap L)$. Elements of $L_1$ are incomparable. By ($\star$), elements of $L_1\setminus B$ are incomparable with elements of $H_0$. Hence, it is for $x\in L_1\cap L$ and $y\in H_0\cap L$ that $x<y$ or $y\le x$. Since $x$ is in $L$, not $y\le x$. Since $y$ is not in $H$, there is a $z\in H_1\cup (L_1\cap C)$ such that $y\le z$, so $x<z$. This contradicts how $L_1\times \{0\}\cup H_1\times \{1\}$ is an antichain. Our claim follows. \end{proof}

\section{Acknowledgements}

The author is financially supported by a generous USAMO prize award from the Akamai Foundation, and would like to thank Prof Imre Leader for highlighting this conjecture. Furthermore, although it has resolved a small percentage of open problems in combinatorics as of yet, the author commends ChatGPT 5.6 Sol Pro for its transformational impact and for generating the proof content of this note.

\textsc{Independent author in Szeged, Hungary.} \\ \\
\textit{E-mail address:} \texttt{kkw25@cantab.ac.uk}

\end{document}